\documentclass[11pt]{amsart}
\usepackage{amsmath,amssymb,mathrsfs}
\usepackage{pstricks,pst-plot}
\usepackage[small,heads=LaTeX,nohug]{diagrams}
\usepackage[colorlinks=true,linkcolor=black,citecolor=black,urlcolor=black]{hyperref}

\psset{unit=1pt}
\psset{arrowsize=4pt 1}
\psset{linewidth=.5pt}

\newcommand{\define}{\textbf}

\newcommand{\ul}{\underline}

\newcommand{\isom}{\cong}
\renewcommand{\setminus}{\smallsetminus}
\renewcommand{\phi}{\varphi}

\renewcommand{\tilde}{\widetilde}
\renewcommand{\hat}{\widehat}
\renewcommand{\bar}{\overline}

\newcommand{\Q}{\mathbb{Q}}
\newcommand{\R}{\mathbb{R}}
\newcommand{\Z}{\mathbb{Z}}
\newcommand{\PP}{\mathbb{P}}
\newcommand{\A}{\mathbb{A}}
\newcommand{\G}{\mathbb{G}}

\newcommand{\OO}{\mathcal{O}}

\newcommand{\cL}{\mathcal{L}}

\newcommand{\liet}{\mathfrak{t}}

\newcommand{\XX}{\mathfrak{X}}

\newcommand{\Fl}{{Fl}}

\newcommand{\pr}{\mathrm{pr}}

\DeclareMathOperator{\Aut}{Aut}
\DeclareMathOperator{\Pic}{Pic}
\DeclareMathOperator{\Bl}{Bl}

\newcommand{\Eff}{\overline{\mathrm{Eff}}}
\newcommand{\Nef}{\mathrm{Nef}}

\newtheorem{theorem}{Theorem}[section]
\newtheorem{lemma}[theorem]{Lemma}
\newtheorem{proposition}[theorem]{Proposition}
\newtheorem{corollary}[theorem]{Corollary}

\theoremstyle{definition}

\newtheorem{remark}[theorem]{Remark}
\newtheorem{example}[theorem]{Example}

\begin{document}

\title{Effective divisors on Bott-Samelson varieties}
\author{Dave Anderson}
\address{Department of Mathematics, The Ohio State University, Columbus, OH 43210}
\email{anderson.2804@math.osu.edu}
\date{October 21, 2016}
\thanks{The author was partially supported by NSF Grants DMS-0902967 and DMS-1502201, as well as a postdoctoral fellowship from the Instituto Nacional de Matem\'atica Pura e Aplicada (IMPA)}

\begin{abstract}
We compute the cone of effective divisors on a Bott-Samelson variety corresponding to an arbitrary sequence of simple roots.  The main tool is a general result concerning effective cones of certain $B$-equivariant $\PP^1$ bundles.  As an application, we compute the cone of effective codimension-two cycles on Bott-Samelson varieties corresponding to reduced words.  We also obtain an auxiliary result giving criteria for a Bott-Samelson variety to contain a dense $B$-orbit, and we construct desingularizations of intersections of Schubert varieties.  An appendix exhibits an explicit divisor showing that any Bott-Samelson variety is log Fano.
\end{abstract}

\maketitle

\section{Introduction}\label{s.intro}

The cone of pseudoeffective divisors carries important information about a projective algebraic variety, and plays a key role in the minimal model program.  Computing this cone explicitly, however, is quite difficult; there are few general methods for describing effective divisors.  In general, the pseudoeffective cone may have infinitely many extremal rays; but even for log Fano varieties, where one knows the cone is finitely generated \cite{bchm}, describing a set of generators is a challenging task.

One situation in which the answer is relatively simple is common in applications to representation theory.  Whenever a connected solvable group $B$ acts on a nonsingular projective variety $X$, one knows the effective cone is generated by $B$-invariant divisors: given a $B$-linearized line bundle $\cL$, apply the Lie-Kolchin theorem to $H^0(X,\cL)$ to find a nonzero semi-invariant section whenever $\cL$ has nonzero sections; such a section defines a $B$-invariant divisor.  Sumihiro's theorem guarantees that any line bundle on $X$ can be linearized, so this suffices to prove the claim.  (See \cite{sumihiro}, and the refinement in \cite[Proposition~2.4]{kklv}.  The same reasoning also applies to $\Q$-factorial normal varieties.)  
In the case when $B$ acts with a dense open orbit $U$ whose complement has codimension one, the irreducible components of $X\setminus U$ are precisely the irreducible $B$-invariant divisors, so these generate $\Eff(X)$, possibly redundantly.  Examples of this include toric varieties and flag varieties---and more generally, spherical varieties---as well as Schubert varieties and their equivariant desingularizations.

It is harder to make such general statements for varieties not having a group action with a dense orbit, but the computation of effective cones has become a subject of increasing activity in the last decade, especially in connection with birational geometry and Cox rings.  
The main goal of this article is to describe the effective cones of divisors on Bott-Samelson varieties.  
These include desingularizations of Schubert varieties, but in general they do not admit solvable group actions with dense orbits (Remark~\ref{r.nodense}).  
Bott-Samelson varieties form an infinite family of varieties whose effective cones are of unbounded dimension, and these cones can be described concretely; the main difficulty lies in identifying the correct statement.

In order to state the theorem we quickly review the setup; more can be found in \S\ref{s.bs}.  Consider a reductive (or Kac-Moody) group $G$ with Borel subgroup $B$.  For any sequence $\ul\alpha=(\alpha_1,\ldots,\alpha_d)$ of simple roots and for $i\leq d$, write $w_i = s_{\alpha_1}\star \cdots \star s_{\alpha_i}$ for the Demazure product of the corresponding simple reflections.  The Bott-Samelson variety $X(\ul\alpha)$ comes with $B$-equivariant maps $\phi_i\colon X(\ul\alpha) \to G/B$ for $1\leq i\leq d$, and the image of each $\phi_i$ is the Schubert variety $X(w_i)$.  (We will often write $w=w_d$ and $\phi=\phi_d$.)  The variety $X(\ul\alpha)$ also comes with $d$ ``standard'' $B$-invariant divisors $X_1,\ldots,X_d$, and their classes form a basis for the Picard group (and N\'eron-Severi group).

The sequence $\ul\alpha$ is {\em reduced} if the Demazure product is equal to the ordinary product in the Weyl group; in this case, the map $\phi\colon X(\ul\alpha) \to X(w)$ is birational, and $X(\ul\alpha)$ has a dense $B$-orbit whose complement is the union of the standard divisors $X_i$.  As already pointed out, this means the effective cone of a Bott-Samelson variety $X(\ul\alpha)$ is generated by the standard divisors when $\ul\alpha$ is reduced.  
In general, however, there are effective classes on $X(\ul\alpha)$ not contained in the simplicial cone spanned by $X_1,\ldots,X_d$.

Let $X = X(\ul\alpha)$.  For each $i$, there is a truncation map $\pi_i\colon X \to X(\alpha_1,\ldots,\alpha_i)$.  There is always a unique irreducible component of the 
divisor $\phi_{i}^{-1}X(w_{i}s_{\alpha_{i}})$ which projects surjectively onto $X(\alpha_1,\ldots,\alpha_{i-1})$ via $\pi_{i-1}$.  When $w_{i}  > w_{i-1}$, this divisor is $X_{i}$; when $w_{i} = w_{i-1}$, it is a distinct divisor $\Sigma_{i}$ (Lemma~\ref{l.component}).  The latter case occurs for some $i$ if and only if the word $\ul\alpha$ is non-reduced.

\begin{theorem}\label{t.main}
The effective cone $\Eff(X)$ is generated by the standard divisors $X_1,\ldots,X_d$, together with the divisors $\Sigma_i$ for $i$ such that $w_{i} = w_{i-1}$.  Furthermore, each $\Sigma_i$ spans an extremal ray.
\end{theorem}

The generating set is not claimed to be minimal: when $w_{i}=w_{i-1}$, the divisor $X_i$ may be redundant.  In the simplest example, $X=X(\alpha,\alpha)\isom \PP^1\times\PP^1$, the effective cone is generated by $X_1$ and $\Sigma_2 = X_2-X_1$.  We will give examples showing that the number of rays of $\Eff(X)$ can be larger than $d=\dim NS(X)_\R$.  This shows that the cones can be non-simplicial, but nonetheless they are not too complicated: the theorem implies a bound of $2d-1$ extremal rays.  (This is close to sharp, since one can find Bott-Samelson varieties of dimension $d$ whose effective cones have $2d-3$ extremal rays; see Example~\ref{ex.a2one}.)

Even when one is primarily interested in reduced words, non-reduced words appear naturally. The standard divisors $X_i\subseteq X(\ul\alpha)$ are themselves isomorphic to Bott-Samelson varieties, and many of them correspond to non-reduced words; having control over their effective cones is useful for running inductive arguments.  For instance, Theorem~\ref{t.main} gives a concrete way to compute the cone of effective cycles of codimension two in $X(\ul\alpha)$, when $\ul\alpha$ is reduced.

\begin{theorem}\label{t.main2}
Suppose $\ul\alpha$ is reduced, and let $X_1,\ldots,X_d$ be the $B$-invariant divisors in $X=X(\ul\alpha)$.  Consider each $\Eff(X_i)$ as a subcone of $\Eff_{d-2}(X)$ via the pushforward $N_{d-2}(X_i)_\R \to N_{d-2}(X)_\R$.  Then $\Eff_{d-2}(X) = \sum_{i=1}^d \Eff(X_i)$, and the right-hand side is computed as in Theorem~\ref{t.main}.  (In particular, $\Eff_{d-2}(X)$ is finitely generated.)
\end{theorem}

\noindent
The proof is immediate: any effective class in $N_{d-2}(X)$ can be represented by an effective $B$-invariant cycle (see Lemma~\ref{l.binvt}), and any irreducible $B$-invariant subvariety lies inside an irreducible $B$-invariant divisor.  Since $\ul\alpha$ is reduced, the only such divisors in $X$ are the $X_i$, so a codimension-two subvariety of $X$ is an effective divisor in $X_i$, to which Theorem~\ref{t.main} applies.

Theorem~\ref{t.main2} is of particular interest, since little is known about positivity for higher-codimension cycles.  Compared to the nef and effective cones of divisors and curves, the higher-codimension analogues of these notions can exhibit pathological behavior.  Debarre, Ein, Lazarsfeld, and Voisin exhibited examples of abelian varieties with nef classes which are not pseudoeffective \cite{delv}, and Ottem has recently found examples of this phenomenon on hyperk\"ahler fourfolds \cite{ottem}.  
These pathologies do not occur for four-dimensional Bott-Samelson varieties, though.  In \S\ref{s.exs}, a typical example is described in detail, with explicit generators for the nef and effective cones of $2$-cycles.

Non-reduced words also come up in several other contexts.  
If $\ul\alpha$ and $\ul\beta$ are two reduced words, then certain fibers of the map $X(\ul\alpha,\ul\beta) \to G/B$ give desingularizations of certain intersections of Schubert varieties; these varieties are nontrivial exactly when the concatenation $(\ul\alpha,\ul\beta)$ is non-reduced (see \S\ref{s.richardson}).  As another example, when $G=SL_3$, Bott-Samelson varieties for arbitrary sequences of roots are also certain configuration spaces for points and lines in the plane (see Remark~\ref{r.config}).

Theorem~\ref{t.main} is deduced as a direct consequence of a result which compares the pseudoeffective cones of $Z(\beta)$ and $Z$, where $Z$ is any smooth variety with a $B$-equivariant map to $G/B$, and $Z(\beta) \to Z$ is a certain $B$-equivariant $\PP^1$ bundle corresponding to a simple root $\beta$ (Theorem~\ref{t.mainZ}).  In \S\ref{s.b-orbit} and \S\ref{s.richardson}, we give two supplemental results: one is a criterion for $X(\ul\alpha)$ to have a dense $B$-orbit, and the other constructs a desingularization of the (often non-transverse) intersection $X(u) \cap w\cdot X(v)$ of translated Schubert varieties.

\medskip
\noindent
{\it Acknowledgements.}  I thank Ana-Maria Castravet, Mihai Fulger, Emanuele Macr\`i, Edward Richmond, Jenia Tevelev, and Jesper Funch Thomsen for helpful comments and conversations.

\section{Preliminaries and notation}\label{s.prelim}

Fix an algebraically closed ground field.

\medskip

We refer to \cite{kumar} for basic facts about groups and flag varieties.  Throughout, $G$ denotes a Kac-Moody group, with a fixed Borel subgroup $B$ and maximal torus $T$.  Let $W = N_G(T)/T$ be the corresponding Weyl group, $R = R^+ \cup R^- \subseteq \liet^*$ the set of (positive and negative) roots, and $S\subseteq R^+$ the set of simple roots.  The roots of $B$ are $R^+$.  There is also an opposite Borel subgroup $B^-$, whose roots are $R^-$.

Corresponding to each root $\alpha$ there is a coroot $\alpha^\vee\in \liet$.  For each simple root, there is also a fundamental weight $\varpi_\alpha$, defined so that $\langle \varpi_\alpha,\beta^\vee \rangle = \delta_{\alpha,\beta}$.

For each $\alpha \in S$ there is a simple reflection $s_\alpha$, and a minimal parabolic subgroup $P_\alpha \subseteq G$, generated by $B$ and $\dot{s}_\alpha$; the roots of $P_\alpha$ are $R^+ \cup \{-\alpha\}$.

We will fix representatives $\dot{w}\in N(T)$ for Weyl group elements $w\in W$, but the choice will never matter in this article.  The length $\ell(w)$ of a Weyl group element is the minimal number of simple reflections $s_\alpha$ needed to write $w$.  Such an expression is called {\em reduced}, as is the corresponding sequence of simple roots $\alpha$.

The {\em Demazure product} on $W$ is defined by replacing the relation $s_\alpha\cdot s_\alpha=e$ with $s_\alpha\star s_\alpha=s_\alpha$ in the Coxeter presentation.  Equivalently, for any $w\in W$ and any simple root $\alpha$, define $w\star s_\alpha$ to be the longer of $ws_\alpha$ or $w$.  A sequence of simple roots is reduced if and only if the usual product $s_{\alpha_1}\cdots s_{\alpha_\ell}$ is equal to the Demazure product $s_{\alpha_1}\star \cdots \star s_{\alpha_\ell}$.

The flag variety is $G/B$, an (ind-)projective variety.  It decomposes into Schubert cells $X(w)^\circ = B\dot{w}B/B \isom \A^{\ell(w)}$, whose closures are the Schubert varieties $X(w) = \overline{B\dot{w}B/B}$.  The torus $T$ acts on $G/B$ by left multiplication, and the fixed points $(G/B)^T$ are in bijection with the Weyl group.  Abusing notation, we sometimes write $w\in G/B$ to indicate the $T$-fixed point corresponding to $w\in W$, instead of writing $\dot{w}B$.

For any $w\in W$, we define
\[
  B_w = B \cap \dot{w}B\dot{w}^{-1} \qquad\text{and}\qquad B^w = B \cap \dot{w}B^-\dot{w}^{-1},
\]
and similarly
\[
  U_w = U \cap \dot{w}U\dot{w}^{-1} \qquad\text{and}\qquad U^w = B \cap \dot{w}U^-\dot{w}^{-1},
\]
where $U$ and $U^-$ are the unipotent radicals of $B$ and $B^-$, respectively.  Then $B_w$ is the stabilizer (in $B$) of the $T$-fixed point $\dot{w}B\in G/B$, and the map $u \mapsto u\dot{w}B$ defines a $T$-equivariant isomorphism $U^w \xrightarrow{\sim} X(w)^\circ$.

For each root $\alpha\in R$, there is a root subgroup $U_\alpha \isom \mathbb{G}_a$.  There is an isomorphism of algebraic varieties
\begin{align*}
\A^{\ell(w)} \isom \prod_{\alpha\in R^+ \cap w(R^-)} U_\alpha \xrightarrow{\sim} U^w ,
\end{align*}
given by multiplication (in any order).  This is equivariant for the conjugation action of $T$ on both sides.

The reader less familiar with Lie theory is invited to assume $G=SL_n$, with $B$ and $T$ the subgroups of upper-triangular and diagonal matrices.  The Weyl group is $W=S_n$, the simple reflections are the adjacent transpositions, the simple roots are $\alpha_i = e_i-e_{i+1}$, and the minimal parabolic $P_{\alpha_i}$ consists of block upper-triangular matrices where an entry in position $(i+1,i)$ can be nonzero.  (Theorem~\ref{t.main} is already nontrivial for $SL_3$.)

\medskip
Finally, we record a basic fact about invariant cycles.  
Let $B$ be a connected solvable linear algebraic group acting on a smooth projective variety $X$.  As mentioned in the introduction, $B$-invariant divisors suffice to generate the cone of effective divisors.  Substantially more is true.

\begin{lemma}\label{l.binvt}
Let $X$ be a complete irreducible variety of dimension $d$, with an action of the connected solvable group $B$.  Then for any $0\leq i\leq d$, the cone of effective classes in $N_i(X)_\R$ is generated by classes of $B$-invariant $i$-cycles.
\end{lemma}

\noindent
This is a straightforward generalization of statements in \cite{brion,fmss}.  The idea in all cases is to apply Borel's fixed point theorem to the Chow variety of $X$.

\section{Divisors on $B$-varieties}\label{s.bdiv}

In this section we will consider a complete nonsingular $B$-variety $Z$, together with a $B$-equivariant map $\phi\colon Z \to G/B$, and compute the effective cone of certain $\PP^1$-bundles on $Z$.

Being a closed irreducible $B$-invariant subvariety of $G/B$, the image $\phi(Z)\subseteq G/B$ is a Schubert variety $X(w)$.  The following simple lemma will be useful.

\begin{lemma}\label{l.smooth}
Let $\phi \colon Z \to X(w) \subseteq G/B$ be as above.  The fiber $\phi^{-1}(w)$ is nonsingular and irreducible.
\end{lemma}

\begin{proof}
Let $Z^\circ = \phi^{-1}X(w)^\circ$.  Since $X(w)^\circ$ is a principal homogeneous space for $U^w\subseteq B$, and the map $Z^\circ \to X(w)^\circ$ is equivariant, we have $Z^\circ \isom \phi^{-1}(w) \times X(w)^\circ$.  The subset $Z^\circ$ is open in $Z$, so it is nonsingular and irreducible; since $X(w)^\circ$ is also nonsingular, the claim follows. \qed
\end{proof}

We fix some notation.  For any simple root $\beta$, let $P_\beta\subset G$ be the corresponding minimal parabolic subgroup, so there is a $\PP^1$-bundle $G/B \to G/P_\beta$.  Write
\[
  \XX(\beta) = G/B \times_{G/P_\beta} G/B
\]
for the fiber product.  

Define $\pi\colon Z(\beta) \to Z$ to be the pullback of the $\PP^1$-bundle $\XX(\beta) \to G/B$.  Thus there is a diagram
\begin{diagram}
  Z(\beta) & \rTo & \XX(\beta) & \rTo & G/B \\
   \dTo^\pi    &      &  \dTo      \\
   Z       & \rTo & G/B.
\end{diagram}
There is a canonical section of the projection $\XX(\beta) \to G/B$, corresponding to the diagonal map $G/B \to G/B \times G/B$.  This induces a canonical section of $\pi$.  We will write $\Delta\subseteq Z(\beta)$ for the image of this section; it is a $B$-invariant divisor.

Writing $X(w,\beta)$ for the restriction of $\XX(\beta)$ to $X(w)\subseteq G/B$, we also have a diagram
\begin{equation}\label{e.basic1}
\begin{diagram}
  Z(\beta) & \rTo^\psi & X(w,\beta) & \rTo^{\mu} & X(w\star s_\beta) \\
   \dTo^\pi    &      &  \dTo_{\bar\pi}      \\
   Z       & \rTo^\phi & X(w),
\end{diagram}
\end{equation}
where $w\star s_\beta$ is the Demazure product (equal to the longer of $w s_\beta$ or $w$).  Let $\tilde\phi\colon Z(\beta) \to X(w\star s_\beta)$ be the composition $\mu\circ\psi$.

Restricting this basic diagram over $X(w)^\circ \subseteq X(w)$, we have
\begin{equation}\label{e.basic2}
\begin{diagram}
 Z^\circ\times E & \rTo & X(w)^\circ\times E & \rTo^{\mu^\circ} &  X(w)^\circ \cup X(ws_\beta)^\circ  \\
   \dTo^{\pi^\circ}          &      &  \dTo      \\
 Z^\circ  & \rTo & X(w)^\circ,
\end{diagram}
\end{equation}
where $E = E_{w,ws_\beta}\isom\PP^1$ is the curve connecting the $T$-fixed points $w$ and $ws_\beta$ in $G/B$, and as already observed in the proof of Lemma~\ref{l.smooth}, $Z^\circ = \phi^{-1}X(w)^\circ$ is isomorphic to $X(w)^\circ \times \phi^{-1}(w)$.  

These two diagrams will be used repeatedly in proving the main results.  It will be helpful to describe \eqref{e.basic2} in more detail first.

The isomorphism $\bar\pi^{-1}X(w)^\circ = X(w)^\circ \times E$ is given by
\[
  {[ u\cdot w, p ]} \mapsto (u\dot{w}B, \dot{w}pB),
\]
where $u\in U^w$ gives a unique representative $u\dot{w}B$ for a point in $X(w)^\circ$, and $p\in P_\beta$.  The map $\mu^\circ\colon X(w)^\circ \times E \to X(w)^\circ \cup X(ws_\beta)^\circ \subseteq X(w\star s_\beta)$ sends $(u \dot{w}B, \dot{w}pB)$ to $u\dot{w}pB$.

The morphism $\mu^\circ$ can be described more precisely.  Assume for the moment that $w>ws_\beta$, so $w(-\beta) \in R^+$.  For $u'\in U^{ws_\beta}$ and $u''\in U_{w(-\beta)}$, there is an isomorphism
\[
 m\colon X(ws_\beta)^\circ \times U_{w(-\beta)} \times E \xrightarrow{\sim} X(w)^\circ \times E
\]
defined by $m(u'\dot{w}\dot{s}_\beta B, u'', \dot{w}\dot{s}_\beta p'B) = (u'u''\dot{w}B, (u'')^{-1}\dot{w}pB)$, where $p=\dot{s}_\beta p'$.  (This uses the facts that $U^w \isom U^{ws_\beta} \times U_{w(-\beta)}$ and that $U_{w(-\beta)}$ acts on $E$.)

These morphisms fit into a commutative diagram
\begin{equation}\label{e.mu}
\begin{diagram}
 X(w)^\circ \times E & \rTo^{\mu^\circ} & X(w)^\circ \cup X(ws_\beta)^\circ \\
    \uTo_\sim^m  &   &          \uTo_{\bar{\mu}^\circ}^\sim \\
 X(ws_\beta)^\circ \times U_{w(-\beta)} \times E  & \rTo^q  & X(ws_\beta)^\circ \times E,
\end{diagram}
\end{equation}
where $q$ is the projection.  
It follows that $\mu^\circ$ is a (trivial) $\A^1$-bundle when $w>ws_\beta$.  Similar reasoning shows $\mu^\circ$ is an isomorphism when $w<ws_\beta$.

Finally, the action of $B$ on $X(w)^\circ \times E$ can be described as follows.  Writing a point in $X(w)^\circ$ as $u\dot{w}B$, for any $b\in B$, we can write $b=u'b'$ where $u' \in uU^w u^{-1}$ and $b'\in uB_w u^{-1}$.  For $y=\dot{w}pB\in E$, we have
\begin{align*}
  b\cdot (u\dot{w}B, y) &= b\cdot[ u\cdot w, p] \\
  &=  [b u\cdot w, p] \\
  &= [ u'u\cdot w, u^{-1}b'u p] \\
  &= (u'u\dot{w}B, \dot{w}u^{-1}b'upB) \\
  &= (u'u\dot{w}B, b''\dot{w}pB) \\
  &=  (u' u\cdot w, b''\cdot y),
\end{align*}
noting that $u^{-1}b'u \in B_w$, and hence so is $b'' = \dot{w}u^{-1}b'u\dot{w}^{-1}$.

\begin{lemma}\label{l.binv}
Suppose $ws_\beta>w$, let $D\subseteq Z(\beta)$ be a $B$-invariant effective divisor, and let $D^\circ = D \cap (Z^\circ \times E)$ be its restriction to diagram \eqref{e.basic2}.  Then the general fibers of the projection $D^\circ \to Z^\circ$ do not contain $ws_\beta\in E$.
\end{lemma}

\begin{proof}
If the general fiber of $D^\circ \to Z^\circ$ contains $ws_\beta$, then $D^\circ$ must contain $Z^\circ \times \{ws_\beta\}$ as an irreducible component.  By the above calculation, a divisor in $Z^\circ \times E$ of the form $Z^\circ \times \{y\}$ is $B$-invariant if and only if $y$ is fixed by $B_w$.  When $ws_\beta>w$, we have $E^{B_w} = \{w\}$ since in this case the root subgroup $U_{w(\beta)}$ is contained in $B_w$ but not in $B_{ws_\beta}$.\qed
\end{proof}

The standard $B$-invariant divisor $\Delta\subseteq Z(\beta)$ restricts to $\Delta^\circ = Z^\circ\times\{w\}$.  When $ws_\beta<w$, we have $E^{B_w} = \{w,ws_\beta\}$, so we can define a second $B$-invariant divisor
\begin{equation}\label{e.sigma}
  \Sigma = \overline{ Z^\circ \times \{ws_\beta\} } \subseteq Z(\beta).
\end{equation}
Using the $B$-action, this can also be written as $\Sigma = \overline{B\cdot (\phi^{-1}(w)\times \{ws_\beta\})}$. 
The main theorem of this section says that the effective cone of $Z(\beta)$ is generated by that of $Z$, together with $\Delta$ and $\Sigma$ (when the latter exists).

\begin{theorem}\label{t.mainZ}
\begin{enumerate}
\item If $ws_\beta>w$, then $\Eff(Z(\beta))$ is generated by $\pi^*\Eff(Z)$ together with $\Delta$.

\item If $ws_\beta<w$, then $\Eff(Z(\beta))$ is generated by $\pi^*\Eff(Z)$ together with $\Sigma$ and $\Delta$.  In this case $\Sigma$ generates an extremal ray.
\end{enumerate}
\end{theorem}

Quite generally, if $Z$ is a $\Q$-factorial variety and $\pi\colon Z' \to Z$ is a $\PP^1$-bundle, then the N\'eron-Severi groups of $Z$ and $Z'$ are related by a natural short exact sequence $0 \to NS(Z) \xrightarrow{\pi^*} NS(Z') \to \Z\to 0$, and the pullback $\pi^*\Eff(Z)$ is a facet of $\Eff(Z')$.  (A fiber of $\pi$ is a moving curve, and $\pi^*\Eff(Z)$ lies in the orthogonal space to the class of a fiber; in fact, the same argument applies to any fibration whose fiber has Picard number one.)  
Bearing this in mind, Theorem~\ref{t.mainZ} is an immediate consequence of the following proposition, together with Proposition~\ref{p.sigma-extremal}.

\begin{proposition}\label{p.degenerate}
Let $D\subseteq Z(\beta)$ be a $B$-invariant reduced and irreducible effective divisor.  Then $D$ is rationally equivalent to $a\Delta+b\Sigma+ \pi^*D'$ for $a,b\geq 0$ and some effective divisor $D'\subseteq Z$.
\end{proposition}

\begin{proof} 
Let $\lambda\colon \G_m \to T$ be a one-parameter subgroup such that $\langle w(\beta),\lambda \rangle >0$, so the induced action of $\G_m$ act on $E$ fixes the points $w$ and $ws_\beta$, and has $\lim_{t\to 0} t\cdot y = w$ for $y\in E\setminus\{w,ws_\beta\}$.  Extend this to an action of $\G_m$ on $Z^\circ \times E$ by letting $\G_m$ act trivially on $Z^\circ$.

Define a family
\[
  \mathcal{D}^* := \{( t\cdot x, t)  \,| \, x\in D^\circ\} \subseteq (Z^\circ\times E) \times \G_m,
\]
and let $\mathcal{D}$ be the (reduced) closure in $Z(\beta) \times \A^1$.  Letting $\G_m$ act on $\A^1$ by extending multiplication on itself, $\mathcal{D}$ is a $\G_m$-invariant divisor in $Z(\beta)\times \A^1$.  For each $t\in\A^1$, write $D_t$ for the fiber of the projection $p\colon \mathcal{D} \to \A^1$; there is a map
\[
 \pi_t \colon D_t \to Z,
\]
which is identified with $\pi \colon D \to Z$ for $t=1$.

Let $\mathcal{D}^\circ = \mathcal{D}\cap (Z^\circ \times E) \times \A^1$, so we have a projection $p^\circ \colon \mathcal{D}^\circ \to \A^1$ with fibers $D^\circ_t = (p^\circ)^{-1}(t)$, as well as maps $\pi^{\circ}_t\colon D^\circ_t \to Z^\circ$.  Note that $D^\circ_t = D_t \cap (Z^\circ \times E)$.  The limit $D_0^\circ$ is a $\G_m$-invariant divisor in $Z^\circ \times E$, and therefore a fiber of $\pi^\circ_0$ is either all of $E$, or else it is contained in $\{w,ws_\beta\} \subseteq E$.

We distinguish three cases for irreducible components of $D_0$.  A component of $D_0$ whose general fiber is $w\in E$ must be equal to $\Delta$.  A component whose general fiber is $ws_\beta\in E$ must be equal to $\Sigma$.  (By Lemma~\ref{l.binv}, the latter case occurs only if $ws_\beta<w$.)  The general fiber of $\pi^\circ_0$ cannot be all of $E$, so the third possibility is that the general fiber is empty. 
A component of $D^\circ_0$ which does not surject onto $Z^\circ$ maps into a $B$-invariant divisor of $Z$; and any component of $D_0$ that is not in $\overline{D^\circ_0}$ must be contained in $\pi^{-1}(Z\setminus Z^\circ)$, so such a component also maps into a $B$-invariant divisor of $Z$.  
It follows that $D_0$ is supported on $\Delta + \Sigma + \pi^{-1}D'$ for some $B$-invariant effective divisor $D'\subseteq Z$, with $\Sigma$ appearing only if $ws_\beta>w$.

The scheme $\mathcal{D}$ is Cohen-Macaulay, since it is a divisor in the nonsingular variety $Z(\beta)\times\A^1$, so equidimensionality of the fibers implies $p$ is flat (e.g. by \cite[Ex.~III.10.9]{hartshorne}).  This proves the proposition.\qed
\end{proof}

The divisor $\Sigma$ has a useful alternative characterization.  In the top row of diagram \eqref{e.basic2}, observe that the image of $\phi^{-1}(w)\times \{ws_\beta\}$ under the composed map $Z^\circ \times E \to X(w)^\circ \times E \to X(w)^\circ \cup X(ws_\beta)^\circ$ is the point $ws_\beta$.  By $B$-equivariance, when $ws_\beta<w$, we see that $\tilde\phi$ maps $\Sigma$ into the divisor $X(ws_\beta)\subseteq X(w)$.

\begin{lemma}\label{l.component}
If $w s_\beta < w$, then $\Sigma$ is the unique irreducible component of the divisor $\tilde\phi^{-1}X(ws_\beta) \subseteq Z(\beta)$ which is mapped birationally to $Z$ by $\pi$.  Furthermore, the scheme $\tilde\phi^{-1}X(ws_\beta)$ is generically reduced along $\Sigma$.
\end{lemma}

\begin{proof}
Suppose $\Theta$ is an irreducible component of $\tilde\phi^{-1}X(ws_\beta)$ mapping birationally to $Z$.  In the notation of the basic diagram \eqref{e.basic1}, there are exactly two irreducible $B$-invariant divisors in $X(w,\beta)$ which are mapped birationally to $X(w)$ by $\bar\pi$, namely the standard divisor $\Delta' = \overline{B\cdot [w,e]}$ and the orbit closure $\Sigma' = \bar{B\cdot [w,s_\beta]}$.  (For any point $y$ other than $e$ or $s_\beta$, the orbit $B\cdot[w,y]$ is dense in $X(w,\beta)$.)  Since $\Theta$ is irreducible, we have $\Theta \subseteq \psi^{-1}\Delta'$ or $\Theta\subseteq\psi^{-1}\Sigma'$.  Since $\bar\pi$ restricted to either $\Delta'$ or $\Sigma'$ is birational, we must have $\psi(\Theta) = \Delta'$ or $\psi(\Theta) = \Sigma'$, respectively.  The former case cannot hold for $\Theta \subseteq \tilde\phi^{-1}X(ws_\beta)$, since $\mu(\Delta')=X(w)$.  Therefore $\Theta \subseteq \psi^{-1}(\Sigma')$.

Now restrict to the basic diagram \eqref{e.basic2}, and consider the dense open set $V = B\cdot [w,s_\beta] =X(w)^\circ \times \{ws_\beta\} \subseteq \Sigma'$.  Since $\bar\pi$ is $B$-equivariant, its restriction to $V$ is an isomorphism onto its image in $X(w)$, which is the cell $X(w)^\circ$.  From the definition, we have $\psi^{-1}V \subseteq \Sigma$.  On the other hand, an open set of $\Theta$ must map onto $V$, so it follows that  $\Theta = \bar{\psi^{-1}V} = \Sigma$.

To see that $\tilde\phi^{-1}X(ws_\beta)$ is reduced along $\Sigma$, again restrict to the basic diagram \eqref{e.basic2}.  Recall that  $\mu^\circ \colon X(w)^\circ \times E \to X(w)^\circ \cup X(ws_\beta)^\circ$ is given by $\mu^\circ( uwB, wpB ) = uwpB$, and by the diagram \eqref{e.mu}, it is identified with the projection of a trivial $\A^1$-bundle.  It follows that $(\mu^\circ)^{-1}X(ws_\beta)^\circ$ is reduced, and one checks that it is $V = X(w)^\circ \times \{ws_\beta\}$ of the previous paragraph.  
Since $\psi^{-1}V = Z^{\circ} \times \{ws_\beta\}$ is a dense open set in $\Sigma$, we are done.\qed
\end{proof}

Proposition~\ref{p.degenerate} shows that $\Eff(Z(\beta))$ is generated by $\pi^*\Eff(Z)$, $\Delta$, and $\Sigma$ (when the latter exists).  To complete the proof of Theorem~\ref{t.mainZ}, we must show $\Sigma$ is extremal when it exists.  We need another simple lemma.

\begin{lemma}\label{l.normal}
Let $X$ be a normal variety, and let $Y\subseteq X$ be a prime (Weil) divisor, i.e., an irreducible subvariety of codimension one.  Let $f\colon X' \to X$ be a proper surjective morphism whose generic fiber is connected and reduced, where $X'$ is a nonsingular variety.  Write $Y' = f^{-1}Y \subseteq X$ for the (scheme-theoretic) inverse image, a divisor in $X'$.  
Then $f_*\OO_{X'}(Y') \isom \OO_X(Y)$ as sheaves of $\OO_X$-modules.  In particular, this sheaf is torsion-free.
\end{lemma}

\begin{proof}
If $X$ is nonsingular, this is simply the projection formula, since $\OO_{X'}(Y') \isom f^*\OO_X(Y)$ and $f_*\OO_{X'} \isom \OO_X$.  To reduce to this case, write $\iota\colon X_\mathrm{sm}\hookrightarrow X$ for the inclusion of the smooth locus, define notation by the cartesian diagram
\begin{diagram}
 X'_\mathrm{sm} & \rInto^{\iota'} & X' \\
   \dTo^{f'}  &     & \dTo_f \\
   X_\mathrm{sm} & \rInto^\iota  &  X,
\end{diagram}
and set $Y_\mathrm{sm}=Y\cap X_\mathrm{sm}$.  Then
\begin{align*}
  f_*\OO_{X'}(Y') &\isom f_*\iota'_*(\iota')^*\OO_{X'}(Y') \\
                 &\isom \iota_*f'_*\OO_{X'_\mathrm{sm}}(Y'_\mathrm{sm}) \\
                 &\isom \iota_*\OO_{X_\mathrm{sm}}(Y_\mathrm{sm}) \\
                 &\isom \OO_X(Y),
\end{align*}
as claimed.\qed
\end{proof}

\begin{proposition}\label{p.sigma-extremal}
Suppose $ws_\beta <w$.  Then the divisor $\Sigma$ generates an extremal ray of $\Eff(Z(\beta))$.
\end{proposition}

\begin{proof}
We know that $\Eff(Z(\beta))$ is generated by $\pi^*\Eff(Z)$, $\Delta$, and $\Sigma$.  Since $\Sigma$ maps birationally to $Z$, we also have $\Sigma = \Delta + \pi^*D$ for some (unique) divisor $D$ on $Z$.  To prove the proposition, we will show $D$ is not effective.

By Lemma~\ref{l.component}, we have $\tilde\phi^{-1}X(ws_\beta) = \Sigma + \pi^*E$ for some effective divisor $E\subseteq Z$, so it suffices to show that the divisor
\[
  D' = \tilde\phi^{-1}X(ws_\beta) - \Delta
\]
is not effective.  Let $\cL = \OO_{Z(\beta)}(D')$, and consider the sequence
\[
  0 \to \cL \to \OO(\tilde\phi^{-1}X(ws_\beta)) \to \OO(\tilde\phi^{-1}X(ws_\beta))|_\Delta \to 0
\]
of sheaves on $Z(\beta)$.  We claim that the induced homomorphism
\begin{equation}\label{e.hom}
  \tilde\phi_*\OO(\tilde\phi^{-1}X(ws_\beta)) \to \tilde\phi_*(\OO(\tilde\phi^{-1}X(ws_\beta))|_\Delta)
\end{equation}
is an injection of sheaves on $X(w)$.  
It follows that $H^0(\cL)=0$, so this suffices to prove the proposition.

In fact, both sheaves in \eqref{e.hom} are isomorphic to the divisorial sheaf $\OO_{X(w)}(X(ws_\beta))$.  To see this, observe that $\tilde\phi_*(\OO(\tilde\phi^{-1}X(ws_\beta))|_\Delta) \isom \phi_*\OO_Z(\phi^{-1}X(ws_\beta))$ and apply Lemma~\ref{l.normal} to the morphisms $\tilde\phi\colon Z(\beta) \to X(w)$ and $\phi\colon Z \to X(w)$.  The morphism \eqref{e.hom} is nonzero---it is an isomorphism over the open set $X(w)^\circ$---and a nonzero homomorphism of divisorial sheaves is injective (see, e.g., \cite[Lemma~3.3]{kovacs}), so the claim is proved.\qed
\end{proof}

Theorem~\ref{t.mainZ} shows that the $\PP^1$-bundles $Z(\beta)\to Z$ are special, in the sense that the effective cone of the total space is finitely generated over that of the base.  In general one cannot hope for such simple behavior, even for $T$-equivariant bundles on nice varieties.

\begin{example}\label{ex.infinite}
Suppose $X \to Y$ is a $T$-equivariant $\PP^1$-bundle, with $Y$ a complexity-one $T$-variety.  Even if $\Eff(Y)$ is finitely generated, the cone $\Eff(X)$ need not be.  Indeed, let $Z$ be a $2$-dimensional smooth toric variety corresponding to a fan with $n\geq 11$ rays, having a cone $\sigma$ spanned by rays $\rho_1,\rho_2$ such that $\rho_3,\ldots,\rho_n$ are contained in $-\sigma$.  Assume the ground field is uncountable, and let $E$ be a rank $3$ toric vector bundle on $Z$, chosen as in \cite[Theorem~1.4]{ghps} so that $\PP(E)$ has non-polyhedral effective cone.  Furthermore assume there is an exact sequence
\[
  0 \to L \to E \to F \to 0
\]
of toric vector bundles on $Z$, with $L$ a line bundle and $F$ of rank $2$.  
(To see this can be arranged, start with any $E$ as in \cite{ghps} so that $\Eff(\PP(E))$ is not finitely generated, and choose a $T$-invariant rational section $\sigma\colon Z \dashrightarrow\PP(E)$.  If $\sigma$ is regular, then it corresponds to such an $L$.  Otherwise, blow up finitely many $T$-invariant points to resolve indeterminacy of $\sigma$, obtaining a new toric variety $Z'\to Z$; now replace $Z$ with $Z'$, and $E$ with its pullback to $Z'$.)

The line bundle defines a section $Z \to \PP(E)$; let $S \subset \PP(E)$ be the image, let $X=\Bl_{S}(\PP(E))$, and let $Y=\PP(F)$.  The effective cone of $X$ surjects onto that of $\PP(E)$, so $\Eff(X)$ is not finitely generated.  On the other hand, the map of vector bundles $E \to F$ on $Z$ defines a $\PP^1$-bundle $X \to Y$, and $Y$ being the projectivization of a rank-two bundle on $Z$, $\Eff(Y)$ is finitely generated \cite{knop,hs,gonzalez}.  I thank J.~Tevelev for suggesting this construction.
\end{example}

\section{Bott-Samelson varieties}\label{s.bs}

We briefly review basic facts about Bott-Samelson varieties.  For more details, the articles by Demazure \cite{demazure}, Lauritzen and Thomsen \cite{lt}, Magyar \cite{magyar}, and Willems \cite{willems} are recommended.

Recall that for each simple root $\alpha$, there is a minimal parabolic subgroup $P_\alpha\subseteq G$.  
For a group $\Gamma$ acting on $Y$ on the right and $Z$ on the left, the {\it balanced quotient} is defined as $Y \times^\Gamma Z = (Y\times Z)/(y\cdot g,z)\sim (y,g\cdot z)$.

Given a sequence $\ul\alpha = (\alpha_1,\ldots,\alpha_d)$ of simple roots, the corresponding \textit{Bott-Samelson variety} is
\[
  X(\ul\alpha) = P_{\alpha_1} \times^B P_{\alpha_2}\times^B \cdots \times^B P_{\alpha_d}/B,
\]
where $B$ acts on the parabolic subgroups by multiplication.  
Equivalently, $X(\ul\alpha)$ is the quotient of $P_{\alpha_1} \times \cdots \times P_{\alpha_d}$ by $B^d$, via the action $(p_1,p_2,\ldots,p_d)\cdot(b_1,\ldots,b_d) = (p_1 b_1, b_1^{-1} p_2 b_2,\ldots,b_{d-1}^{-1}p_d b_d)$.

Bott-Samelson varieties come with $B$-equivariant morphisms
\begin{align*}
  \phi_i\colon  X(\ul\alpha) & \to G/B \\
           [p_1,\ldots,p_d] & \mapsto p_1\cdots p_i B
\end{align*}
and
\begin{align*}
  \pi_i \colon X(\ul\alpha) & \to X(\alpha_1,\ldots,\alpha_i) \\
          [p_1,\ldots,p_d] &\mapsto [p_1,\ldots,p_i],
\end{align*}
both for $0\leq i\leq d$.  The product of the $\phi_i$'s defines an isomorphism
\begin{align}
  X(\ul\alpha) & \xrightarrow{\sim} \{eB\}\times_{G/P_{\alpha_1}} G/B \times_{G/P_{\alpha_2}} \cdots \times_{G/P_{\alpha_d}} G/B \label{e.config} \\
  [p_1,\ldots,p_d] & \mapsto (eB,\, p_1B, \,\ldots\,,\, p_1\cdots p_d B) \nonumber,
\end{align}
embedding $X(\ul\alpha)$ in the product $(G/B)^{d+1}$.  With respect to this isomorphism, $\phi_i$ is identified with the projection on the $i$th factor, and $\pi_i$ is identified with projection on the first $i$ factors.

To economize on subscripts, let $\phi = \phi_d$ and $\pi=\pi_{d-1}$.  The projection $\pi$ makes $X=X(\ul\alpha)$ a $\PP^1$-bundle over the smaller Bott-Samelson variety $X(\alpha_1,\ldots,\alpha_{d-1})$, so $X$ is a nonsingular projective variety of dimension $d$, and we have $\Pic(X) \isom NS(X) \isom \Z^d$.

For $1\leq i\leq d$, there is a {\it standard divisor} $X_i \subseteq X$, defined by requiring $p_i=e$.  Evidently we have $X_i \isom X(\alpha_1,\ldots,\hat\alpha_i,\ldots,\alpha_d)$.  The union of the $X_i$ form a normal crossings divisor, and the classes of the $X_i$ form a basis for $NS(X)$.

\begin{proof}[Proof of Theorem~\ref{t.main}]
Using the isomorphism \eqref{e.config}, an equivalent way to define $X(\ul\alpha)$ is by the following recursive procedure: start with $X(\emptyset) = \mathrm{pt}$, with its $B$-equivariant map $\phi$ embedding it as $eB\in G/B$.  In general, we take $Z=X(\alpha_1,\ldots,\alpha_{d-1}) \xrightarrow{\phi} G/B$ as in \S\ref{s.bdiv}; then in notation of that section we have $X(\alpha_1,\ldots,\alpha_{d-1},\alpha_d) = Z(\alpha_d)$.  The divisor $X_d$ is identified with the section $\Delta$, and $\Sigma_d$ is $\Sigma$ (if it exists).  The theorem now follows from Theorem~\ref{t.mainZ} and Lemma~\ref{l.component}.\qed
\end{proof}

\begin{example}
The simplest non-reduced word is $(\alpha,\alpha)$ (for any simple root $\alpha$, in any root system).  As remarked in \cite{lt}, one has $X(\alpha,\alpha) \isom \PP^1\times\PP^1$, and the effective cone is not generated by $X_1$ and $X_2$.  Writing the $B$-fixed point as $0=eB \in \PP^1\isom SL_2/B$, the divisor $X_1$ is $\{0\} \times\PP^1$, and $X_2$ is the diagonal.  The map $\phi$ is projection onto the second factor, so the divisor $\Sigma_2$ is $\phi^{-1}(0) = \PP^1 \times\{0\}$.  The theorem claims that these three divisors generate $\Eff(X(\alpha,\alpha))$, and in fact $X_1$ and $\Sigma_2$ suffice.
\end{example}

For more interesting examples, we will need some facts about line bundles.  Recall that any character $\lambda$ determines a line bundle on $G/B$ by $\cL_\lambda = G \times^B L_\lambda$, where $L_\lambda$ is the one-dimensional representation of $B$ with character $\lambda$.  A character induces line bundles on a Bott-Samelson variety $X=X(\ul\alpha)$ by $\OO_i(\lambda) := \phi_i^*\cL_\lambda$.  Equivalently,
\[
  \OO_i(\lambda) = (P_{\alpha_1} \times \cdots \times P_{\alpha_d})\times^{B^d} L_{\lambda,i},
\]
where for $z\in L_{\lambda,i}$, the action is by $(b_1,\ldots,b_d)\cdot z = \lambda(b_i)z$.

For any simple root $\beta$, the isomorphism $P_\beta/B\isom\PP^1$ identifies the line bundle $\cL_{s_\beta\varpi_\beta}|_{P_\beta/B}$ with $\OO_{\PP^1}(1)$.  Equipping $V_\beta = H^0(P_\beta/B,\cL_{s_\beta\varpi_\beta}|_{P_\beta/B})$ with its natural $B$-action yields exact sequences of $B$-modules,
\[
  0 \to L_{\varpi_\beta} \to V_\beta \to L_{s_\beta\varpi_\beta} \to 0,
\]
and of vector bundles on $G/B$
\[
  0 \to \cL_{\varpi_\beta} \to \mathcal{V}_\beta \to \cL_{s_\beta\varpi_\beta} \to 0,
\]
where $\mathcal{V}_\beta = G \times^B V_\beta$.  Pulling back by $\phi_d$, this induces a sequence on $X$,
\[
  0 \to \OO_d(\varpi_\beta) \to \phi^*\mathcal{V}_\beta \to \OO_d(s_\beta\varpi_\beta) \to 0.
\]
The Bott-Samelson variety $X(\ul\alpha,\beta)$ is then identified with the projective bundle $\PP(\phi_d^*\mathcal{V}_\beta)$.  The corresponding universal quotient bundle is obtained by dualizing the inclusion $\OO_{d+1}(\varpi_\beta) \to \phi_{d+1}^*\mathcal{V}_\beta$, so we define $\OO_{d+1}(1) = \OO_{d+1}(-\varpi_\beta)$, and write $\OO_{d+1}(m)$ for its tensor powers.

The line bundles $\OO_i(1)$ form a basis for $\Pic X(\ul\alpha)$, for $1\leq i\leq d$, and a result of Lauritzen-Thomsen says that $\OO(m_1,\ldots,m_d) :=\OO_1(m_1)\otimes \cdots \otimes \OO_d(m_d)$ is very ample (resp., globally generated) iff all $m_i>0$ (resp., all $m_i\geq0$) \cite{lt}.\footnote{The conventions of \cite{demazure} and \cite{lt} differ from ours: their simple roots are positive for an opposite Borel subgroup.}  It is useful to know a change-of-basis formula for passing to the $X_i$ basis of $\Pic X(\ul\alpha)$.

\begin{lemma}[{cf.~\cite[\S4.2, Proposition 1]{demazure}}] \label{l.expand}
For a character $\lambda$, we have
\[
  \OO_i(\lambda) \isom \OO\left(\sum_{j=1}^i r_{ij}(\lambda) X_j\right)
\]
as line bundles on $X(\ul\alpha)$, where the coefficients are 
\[
  r_{ij}(\lambda) = \langle \lambda, \,s_is_{i-1}\cdots s_j\alpha_j^\vee \rangle = \langle -\lambda, \, s_is_{i-1}\cdots s_{j+1}\alpha_j^\vee \rangle.
\]
\end{lemma}

\noindent
The main ingredient in the proof is the fact that $\OO_d(\lambda)$ has degree $-\langle \lambda,\alpha_d^\vee\rangle$ along the fiber of $\pi\colon X(\alpha_1,\ldots,\alpha_d)\to X(\alpha_1,\ldots,\alpha_{d-1})$ (see \cite[\S2.5, Lemme 2]{demazure}).

\begin{example}\label{ex.a2one}
Consider $G=SL_3$ and the root system of type $A_2$, with simple roots $\alpha$ and $\beta$.  The effective cone of $X=X(\alpha,\beta,\alpha,\beta)$ is generated by
\[
  X_1,\,X_2,\,X_3,\,X_4,\,-X_1+X_3+X_4.
\]
Indeed, $\Sigma_4 = \phi^{-1}X(s_\alpha s_\beta)$.  The Schubert divisor $X(s_\alpha s_\beta) \subseteq SL_3/B$ is a section of the line bundle $\cL_{-\varpi_\beta}$.  We compute the class of $\Sigma$ by using Lemma~\ref{l.expand} to expand $\phi^*\cL_{-\varpi_\beta} = \OO_4(-\varpi_\beta)$ in the $X_i$ basis, obtaining $\Sigma_4 = -X_1+X_3+X_4$.  Since $NS(X)_\R \isom \R^4$, this gives an example of a non-simplicial effective cone.

Extending this computation, for $\ul\alpha$ with $\alpha_i=\alpha$ for $i$ odd and $\alpha_i=\beta$ for $i$ even, one finds that both $X_i$ and $\Sigma_i$ span extremal rays of $\Eff X(\ul\alpha)$ whenever $i>3$, so if $\dim X(\ul\alpha)=d>3$, the effective cone has $2d-3$ extremal rays.
\end{example}

\begin{example}\label{ex.a2two}
Continuing the notation of the previous example (in type $A_2$), consider $X=X(\alpha,\beta,\alpha,\alpha)$.  The effective cone is generated by
\[
  X_1,\,X_2,\,X_3,\,X_4,\,X_1-X_2-X_3+X_4.
\]
To see this, we first compute the expansion of $\phi^*\cL_{-\varpi_\alpha}$ in the $X_i$ basis to obtain $X_1-X_3+X_4$.  However, one checks that $\phi^{-1}X(s_\beta s_\alpha) = X_2 \cup \Sigma_4$ (scheme-theoretically).  Subtracting $X_2$ gives $\Sigma_4 = X_1-X_2-X_3+X_4$.
\end{example}

\begin{example}\label{ex.c2}
Now consider $G=Sp_4$, of type $C_2$, with short root $\alpha$ and long root $\beta$.  (So $\langle \alpha,\beta^\vee \rangle = -1$ and $\langle \beta,\alpha^\vee \rangle = -2$.)  The effective cone of $X(\alpha,\beta,\alpha,\beta,\alpha,\beta)$ is generated by
\begin{align*}
  & X_1,\,X_2,\,X_3,\,X_4,\, X_5,\, X_6,\,\\
   & \quad {-X_1}+X_3+X_4+X_5,\,\, -2X_1-X_2+X_4+2X_5+X_6.
\end{align*}
The last two generators are $\Sigma_5$ and $\Sigma_6$, which are computed by using Lemma~\ref{l.expand} to expand $\OO_5(-\varpi_\alpha)$ and $\OO_6(-\varpi_\beta)$, respectively.
\end{example}

\begin{remark}\label{r.config}
The embedding $\phi_0\times \cdots \times \phi_d\colon X(\ul\alpha) \hookrightarrow (G/B)^{d+1}$ gives the Bott-Samelson variety a configuration space interpretation (cf.~\cite{magyar}).  This is especially vivid in type $A_2$: projectivizing the flags in $SL_3/B=\Fl(3)$, one has configurations of points and lines in $\PP^2$.  From this perspective, a general element of $X(\alpha,\beta,\alpha,\beta)$ looks like
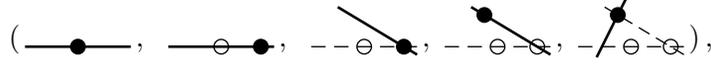
\vspace{.2in}
\[
\left( \mbox{\qquad \psline[linestyle=solid,linewidth=1pt]{-}(-20,0)(20,0) \pscircle*(0,0){3} \qquad},\,
\mbox{ \qquad \psline[linestyle=solid,linewidth=1pt]{-}(-20,0)(20,0) \pscircle(0,0){3} \pscircle*(15,0){3} \qquad },\, 
\mbox{ \qquad \psline[linestyle=dashed]{-}(-20,0)(20,0) \pscircle(0,0){3} \pscircle*(15,0){3} \psline[linestyle=solid,linewidth=1pt]{-}(20,-3)(-10,15)\qquad },\,
\mbox{\qquad \psline[linestyle=dashed]{-}(-20,0)(20,0) \pscircle(0,0){3} \pscircle(15,0){3} \psline[linestyle=solid,linewidth=1pt]{-}(20,-3)(-10,15) \pscircle*(-5,12){3} \qquad },\,
\mbox{\qquad \psline[linestyle=dashed]{-}(-20,0)(20,0) \pscircle(0,0){3} \pscircle(15,0){3} \psline[linestyle=dashed]{-}(20,-3)(-10,15) \pscircle*(-5,12){3} \psline[linestyle=solid,linewidth=1pt]{-}(-2,18)(-13,-4) \qquad } \right),
\]

\noindent
where the first component is the standard flag $eB$.  The $B$-invariant divisors $X_i$ are the loci where the $i$th component agrees with the $(i-1)$st.  The $B$-invariant divisor $\Sigma$ described in Example~\ref{ex.a2one} is the locus where the last component is
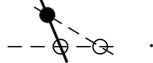
\vspace{.2in}
\[
 \psline[linestyle=dashed]{-}(-20,0)(20,0) \pscircle(0,0){3} \pscircle(15,0){3} \psline[linestyle=dashed]{-}(20,-3)(-10,15) \pscircle*(-5,12){3} \psline[linestyle=solid,linewidth=1pt]{-}(-7.5,18)(2.5,-6) \qquad\quad .
\]

\noindent
That is, the last line passes through the first ($B$-fixed) point.
\end{remark}

\begin{remark}
Every Bott-Samelson variety is log Fano (Theorem~\ref{t.logfano}), so by \cite[Corollary~1.3.2]{bchm}, its Cox ring is finitely generated.  However, $\mathrm{Cox}(X(\ul\alpha))$ usually is not generated by the divisors $X_i$ and $\Sigma_i$ spanning $\Eff(X(\ul\alpha))$.

One already sees counterexamples in the reduced case.  For instance, continuing the notation of the type $A_2$ examples, the variety $X=X(\alpha,\beta)$ is isomorphic to the Hirzebruch surface $\mathbb{F}_1$, a $\PP^1$-bundle over $\PP^1$.  The divisor $X_1$ is a fiber of this bundle. One can also realize $X$ as the blowup of $\PP^2$ in one point, with the divisor $X_2$ identified the exceptional divisor.  We see that $h^0(X,\OO(X_1))=2$ and $h^0(X,\OO(X_2))=1$.  On the other hand,  $X$ is a toric surface with four $T$-invariant divisors, so one knows its Cox ring is isomorphic to $k[t_1,t_2,t_3,t_4]$, which cannot be generated by $H^0(X,\OO(X_1))$ and $H^0(X,\OO(X_2))$.
\end{remark}

\section{Example: nef and effective $2$-cycles}\label{s.exs}

The cone of $i$-dimensional pseudoeffective classes on a smooth projective $d$-dimensional variety $X$ is the closed convex cone $\Eff_i(X) \subseteq N_i(X)_\R$ generated by classes of $i$-dimensional subvarieties.  The cone of nef codimension-$i$ classes is defined as the positive dual of the $i$-dimensional pseudoeffective cone:
\[
  \Nef^i(X) = (\Eff_i(X))^\vee \subseteq (N_i(X)_\R)^* = N_{d-i}(X)_\R.
\]
More details on these and other notions of positivity can be found in \cite{delv,fl}.

In general, little is known about positive cones of higher-codimensional cycles.  Thanks to Theorem~\ref{t.main2}, however, some of these cones can be computed explicitly on Bott-Samelson varieties.  In this section, we will illustrate such an application by studying a particular example.  As in Example~\ref{ex.c2}, let $G=Sp_4$, with short root $\alpha$ and long root $\beta$.  Let $X=X(\alpha,\beta,\alpha,\beta)$; the word is reduced, and the map $\phi\colon X \to G/B$ is birational.

We begin by computing $\Eff_2(X)$.  Let $X_1,X_2,X_3,X_4$ be the standard $B$-invariant divisors.  From the description of $X$ as a tower of $\PP^1$-bundles, one sees that the six surfaces $X_{ij} = X_i \cap X_j$ (for $i\neq j$) form a basis for $N_2(X)$.

\begin{proposition}\label{p.eff2c2}
The cone $\Eff_2(X)$ is simplicial, with six rays generated by
\begin{align*}
 &  v_1 = X_{12},\; v_2 = X_{13},\; v_3 = X_{14},\; v_4 = X_{23}-X_{12}, \\ & \qquad v_5 = X_{24},\; v_6 = X_{34}-X_{23}-2X_{13}.
\end{align*}
\end{proposition}

\begin{proof}
To apply Theorem~\ref{t.main2}, we must compute $\Eff(X_i)$, for $1\leq i\leq 4$.  The cases $i=1$ and $i=4$ correspond to reduced words, so effective divisors are easy to describe: the cones are spanned by $X_{1j}$ ($2\leq j\leq 4$) and $X_{i4}$ ($1\leq i\leq 3$), respectively.  We have $X_2 \isom X(\alpha,\alpha,\beta)$, and Theorem~\ref{t.main} says
\[
  \Eff(X_2) = \langle X_{12},\; X_{23}-X_{12},\; X_{24} \rangle.
\]
Similarly, $X_3 \isom X(\alpha,\beta,\beta)$, and we compute
\[
  \Eff(X_3) = \langle X_{13},\; X_{23},\; X_{34} - X_{23} - 2X_{13} \rangle.
\]
(In the latter case, the computation requires the observation that the Schubert variety $X(s_\alpha)$ is defined by a section of the line bundle $\cL_{2\varpi_\alpha-\varpi_\beta}$, restricted to $X(s_\alpha s_\beta)$.)  The proposition follows.\qed
\end{proof}

To find $\Nef^2(X)$, we must compute the intersection form on $N_2(X)$ with respect to the basis $v_1,\ldots,v_6$.  This is straightforward, using the description of $A^*(X)$ given by Demazure.

\begin{lemma}[{\cite[\S4.2, Proposition 1]{demazure}}]\label{l.ring}
Let $\ul\alpha$ be an arbitrary word, and for $1\leq i\leq d$, let $x_i = [X_i]$ be the divisor class in $A^1(X(\ul\alpha))$.  The product in the Chow (or cohomology) ring is determined by
\begin{align*}
  x_i x_j & = [X_{ij}] \qquad \text{if } i\neq j; \\
  x_j^2   & = -\sum_{i=1}^{j-1} \langle \beta_j, \beta_i^\vee \rangle [X_{ij}],
\end{align*}
where the roots $\beta_i$ are defined by $\beta_i = s_{\alpha_1}\cdots s_{\alpha_{i-1}}(\alpha_i)$.
\end{lemma}

We find that the intersection matrix is
\[
  A = (v_i\cdot v_j) 
  = \left(\begin{array}{rrrrrr}
  0 & 0 & 0 & 0 & 0 & 1 \\
  0 & 0 & 0 & 0 & 1 & -1 \\
  0 & 0 & 0 & 1 & -2 & 1 \\
  0 & 0 & 1 & 0 & -2 & 1 \\
  0 & 1 & -2 & -2 & 4 & -2 \\
  1 & -1 & 1 & 1 & -2 & 2 \end{array}\right),
\]
and its inverse is
\[
  A^{-1} 
  = \left(\begin{array}{rrrrrr}
  0 & 2 & 1 & 1 & 1 & 1 \\
  2 & 4 & 2 & 2 & 1 & 0 \\
  1 & 2 & 0 & 1 & 0 & 0 \\
  1 & 2 & 1 & 0 & 0 & 0 \\
  1 & 1 & 0 & 0 & 0 & 0 \\
  1 & 0 & 0 & 0 & 0 & 0\end{array}\right).
\]
Expressed in the $v$ basis, the rays of the nef cone $\Nef^2(X)$ are the columns of $A^{-1}$.  Translating this into the $X_{ij}$ basis, we can state the result as follows:

\begin{proposition}\label{p.nef2c2}
The nef cone $\Nef^2(X)$ is the simplicial cone whose rays are generated by
\begin{align*}
 &  -X_{12} + X_{14}  + X_{24} + X_{34}, \; \; 
    4X_{13} + 2X_{14} + 2X_{23} + X_{24}, \\
  &  2X_{13} + X_{23}, \;\;  X_{12} + 2X_{13} + X_{14}, \;\;
    X_{12} + X_{13}, \;\;  X_{12};
\end{align*}
All nef classes are effective.
\end{proposition}

Similar calculations show that $\Nef^2(X) \subseteq \Eff_2(X)$ for all four-dimensional Bott-Samelson varieties corresponding to reduced words.  In fact, these all have finitely many $B$-orbits (as can be seen from Proposition~\ref{p.dense} below), so the argument of \cite[Example~4.5]{fl} shows that all nef classes are effective.  It would be interesting to know whether this remains true in higher dimensions.

\section{Dense $B$-orbits}\label{s.b-orbit}

Although we now have a finite list of generators for $\Eff(X(\ul\alpha))$, for any $\ul\alpha$, let us return to the principle sketched in the introduction and ask when $X(\ul\alpha)$ has a dense $B$-orbit.  There is a simple criterion.

Let $w=w(\ul\alpha)$, and choose a subword $(\alpha_{i_1},\ldots,\alpha_{i_\ell})$ that is reduced for $w$, so $w=s_{i_1}\cdots s_{i_\ell}$ and $\ell=\ell(w)$.  Let $m = d-\ell$; this is the dimension of the fiber $\phi^{-1}(w)$.

\begin{proposition}\label{p.dense}
Let $X=X(\ul\alpha)$.  Suppose the $m$ characters
\begin{equation}\label{e.weights}
\begin{array}{c}
  \alpha_1,\ldots,\alpha_{i_1-1}, \\
  s_{i_1}(\alpha_{i_1+1}), \ldots, s_{i_1}(\alpha_{i_2-1}), \\
  s_{i_1}s_{i_2}(\alpha_{i_2+1}), \ldots, s_{i_1}s_{i_2}(\alpha_{i_3-1}), \\
    \vdots \\
  w(\alpha_{i_\ell+1}) , \ldots, w(\alpha_d)
\end{array}
\end{equation}
are linearly independent in $\liet^*$.  Then $X$ has a dense $B$-orbit.
\end{proposition}

The idea is to determine the weights of the $T$-action on the tangent space to a fixed point of $\phi^{-1}(w)$.  The condition on characters ensures that $\phi^{-1}(w)$ is the closure of a $T$-orbit in $X$.

\begin{proof}
From the proof of Lemma~\ref{l.smooth}, we know that $\phi^{-1}X(w)^\circ \isom X(w)^\circ \times \phi^{-1}(w)$.  Since $X(w)^\circ$ is isomorphic to $U^w$, and $B$ is isomorphic to $B_w \times U^w$ (as varieties), $X$ has a dense $B$ orbit if and only if $\phi^{-1}(w)$ has a dense $B_w$-orbit.  Of course, it will suffice to show that $\phi^{-1}(w)$ has a dense $T$-orbit.

Let $\ul\epsilon = (\epsilon_1,\ldots,\epsilon_d)$, with each $\epsilon_i$ being either $e$ or $s_i$ in $W$.  Abusing notation slightly, we also write $\ul\epsilon = [\dot\epsilon_1,\ldots,\dot\epsilon_d]$ for the corresponding point of $X$.  These $2^d$ points are exactly the $T$-fixed points of $X$ (see, e.g., \cite{willems}).  Using the description of $X$ as an iterated $\PP^1$-bundle, one can check that the tangent bundle has a filtration with line bundle factors
\begin{align}\label{e.tangent}
  \OO_1(-\alpha_1),\, \OO_2(-\alpha_2),\, \ldots,\, \OO_d(-\alpha_d),
\end{align}
and so the multi-set of $T$-weights on the tangent space $T_{\ul\epsilon}X$ is
\[
 \Pi_\epsilon = \{ \epsilon_1(-\alpha_1), \, \epsilon_1\epsilon_2(-\alpha_2),\,\ldots,\, \epsilon_1\cdots\epsilon_d(-\alpha_d) \}.
\]

The fixed points of $\phi^{-1}(w)$ are those $\ul\epsilon$ such that $\epsilon_1 \cdots \epsilon_d = w$ (taking the usual product, not the Demazure product).  Take a reduced subsequence for $w$ as above, and let $\ul\epsilon$ be the corresponding point of $\phi^{-1}(w)$, so $\epsilon_{i_j} = s_{i_j}$, and all other $\epsilon_i$'s are $e$.  The $T$-weights on $T_wX(w)^\circ$ are
\[
  R^+ \cap w(R^-) = \{ \alpha_{i_1}, s_{i_1}(\alpha_{i_2}), \ldots, s_{i_1}\cdots s_{i_{\ell-1}}(\alpha_\ell) \},
\]
so the set of weights on $T_{\ul\epsilon}\phi^{-1}(w)$ is the (multi-set) complement $\Pi_{\ul\epsilon} \setminus (R^+ \cap w(R^-))$; these are the negatives of the characters listed in \eqref{e.weights}.  In general, a nonsingular variety $X$ has a dense torus orbit if and only if at some (equivalently, any) fixed point $x$, the weights on $T_xX$ are linearly independent, so the proposition is proved.\qed
\end{proof}

\begin{example}
Continuing the notation for type $C_2$ from Example~\ref{ex.c2}, the criterion shows that $X(\alpha,\beta,\alpha,\beta,\alpha,\beta)$ has a dense $B$-orbit.  In the theorem, take the reduced subword $(\alpha_3,\ldots,\alpha_6)$; the two characters $\alpha_1=\alpha$ and $\alpha_2=\beta$ are certainly linearly independent.  In fact, $w=w_\circ$ in this case, and the fiber $\phi^{-1}(w)$ can be identified with the blowup of $X(s_\alpha s_\beta)$ (which is isomorphic to the Hirzebruch surface $\mathbb{F}_2$) at a torus-fixed point.  The complement of the dense $B$-orbit is the union of the divisors $X_i$ ($1\leq i\leq 6$) together with $\Sigma_5$ and $\Sigma_6$.
\end{example}

It seems reasonable to conjecture that the converse to Proposition~\ref{p.dense} holds.  This would be immediate if $U_w$ acts trivially on $\phi^{-1}(w)$, but I do not know if this is true.

\begin{remark}\label{r.nodense}
One may wonder whether a solvable group larger than $B$ acts on $X=X(\ul\alpha)$, as happens, for instance, when all $\alpha_i$ are the same, so that $X\isom (\PP^1)^d$.  From the definition, it is evident that $P_{\alpha_1}$ acts on $X$, but in fact, usually $\Aut^\circ(X)$ is not too much larger.

The situation can be described more precisely, as follows.  Let us assume $G$ is finite-dimensional, simple, and adjoint, and write $P_i=P_{\alpha_i}$.  The homomorphism $\Pic^\circ((G/B)^{d+1}) \to \Pic^\circ(X)$ induced by the embedding \eqref{e.config} is surjective, and since $\Pic^\circ(X)$ is trivial, $\Aut^\circ(X)$ is affine.  Since a sufficiently divisible power of an ample line bundle can be linearized for $\Aut^\circ(X)$ (by \cite{sumihiro,kklv}), it follows that $\Aut^\circ(X)$ embeds in $\Aut^\circ((G/B)^{d+1})$, and the latter may be identified with $G^{d+1}$ (see \cite{dem-aut}).  Furthermore, writing $\Gamma\subseteq \Aut^\circ((G/B)^{d+1})$ for the subgroup of automorphisms preserving $X$, the homomorphism $\Gamma \to \Aut^\circ(X)$ is surjective.  This group $\Gamma$ can be bounded, using the following proposition.

\begin{proposition}\label{p.aut}
For $1\leq i\leq d$, let $X(i) = X(\alpha_1,\ldots,\alpha_i)$, and let $\Gamma(i)$ be the subgroup of $\Aut^\circ((G/B)^{i+1})\isom G^{i+1}$ preserving $X(i)$. 
If $P_1\cdots P_i=G$, then $\Gamma(i+1)=\Gamma(i)$.
\end{proposition}

\begin{proof}
The groups $\Gamma(i+1)$ can be described inductively as
\[
 \Gamma(i+1) = \left\{ (g_0,g_1,\ldots,g_{i+1}) \;\left|\; (g_0,\ldots,g_i)\in \Gamma(i), \; g_{i+1}^{-1} g_{i} \in \bigcap_{x\in P_1\cdots P_i} x P_{i+1} x^{-1} \right.\right\}.
\]
Indeed,
\[
  X(i+1) \isom (P_1/B) \times_{G/P_2} (P_1P_2/B) \times \cdots \times (P_1\cdots P_i/B) \times_{G/P_{i+1}} (P_1\cdots P_{i+1}/B),
\]
so the condition on  $\Gamma(i+1)$ is that
\begin{align*}
  g_i \cdot x P_{i+1} x^{-1} g_i^{-1} &= g_{i+1} xx_{i+1} P_{i+1} x_{i+1}^{-1}x^{-1} g_{i+1}^{-1} \\
  &= g_{i+1} xP_{i+1}x^{-1} g_{i+1}^{-1},
\end{align*}
that is, $g_{i+1}^{-1}g_i$ normalizes $xP_{i+1}x^{-1}$ for all $x\in P_1\cdots P_i$, or equivalently, $g_{i+1}^{-1}g_i$ lies in $xP_{i+1}x^{-1}$ for all such $x$.  Since $G$ is simple and $P_1\cdots P_i = G$, this intersection on the RHS is the trivial group, so that $g_{i+1}=g_i$.\qed
\end{proof}

In particular, if $\ul\alpha$ is a sufficiently long (non-reduced) word, then no group acts on $X(\ul\alpha)$ with a dense orbit.
\end{remark}

\section{Intersections of Schubert varieties}\label{s.richardson}

The hypothesis of Proposition~\ref{p.dense} says that the action of $T$ makes $\phi^{-1}(w)$ a toric variety with respect to some quotient of $T$.  For a special class of Bott-Samelson varieties, we will see an equivalent alternative that can be checked directly on the flag variety $G/B$.

Given $u,v$ in $W$, let $w = u\star v$ be their Demazure product.  We will consider the intersection $X(u) \cap w\cdot X(v)$ of translated Schubert varieties in $G/B$.  In the finite-dimensional case, if $w=w_\circ$ this is the Richardson variety $X_u^{w_\circ v}$, i.e., the intersection of $X(u)$ with the opposite Schubert variety $Y(w_\circ v) = \overline{B^-w_\circ vB/B}$.  (Conversely, $X(u)\cap Y(w_\circ v)$ is empty unless $u\star v= w_\circ$.)  In the infinite-dimensional case, however, Richardson varieties usually cannot be written as intersections of translated Schubert varieties in this way.\footnote{Ed Richmond has pointed out that in affine type $\hat{A}_1$, given elements $x,y\in W$, one can find $u,v\in W$ such that $X^{y}_{x} = X(u)\cap w\cdot X(v)$ if and only if $\dim X^y_x = \ell(x)-\ell(y)=1$.}

Suppose $\ul\alpha = (\alpha_1,\ldots,\alpha_k)$ is a reduced word for $u$ and $\ul\beta = (\beta_1,\ldots,\beta_\ell)$ is a reduced word for $v$.  Write ${\ul\beta^{-1}}=(\beta_\ell,\ldots,\beta_1)$, so this is a reduced word for $v^{-1}$.

\begin{proposition}\label{p.resolution}
Let $f\colon \phi^{-1}(w) \to G/B$ be the morphism given by restricting $\phi_k$.  Then $f$ is birational onto its image, which is $X(u) \cap w\cdot X(v)$, so $f$ is a desingularization of this intersection of translated Schubert varieties.
\end{proposition}

\noindent
For the case $w=w_\circ$---that is, when $X(u)\cap w\cdot X(v)$ is a Richardson variety---a similar construction was given by Balan \cite{balan} (following a suggestion of Brion), and the geometry of the fiber $\phi^{-1}(w_\circ)$ was studied by Escobar \cite{escobar}.

\begin{proof}
It will be convenient to adopt a more symmetric point of view.  For any sequence of simple roots $\ul\gamma$, write
\[
  \XX(\ul\gamma) = G/B \times_{G/P_{\gamma_1}} \cdots \times_{G/P_{\gamma_d}} G/B,
\]
and let $\pr_0,\ldots,\pr_d$ be the projections onto the $d+1$ factors of $G/B$.  Thus $X(\ul\gamma) = \pr_0^{-1}(e)$, and the maps $\phi_i$ are the restrictions of $\pr_i$.  For any $z\in W$, let us write $zX(\ul\gamma) = \pr_0^{-1}(z)$.  The projection $\pr_d$ maps $zX(\ul\gamma)$ to $z\cdot X(w(\ul\gamma))$, a translated Schubert variety in $G/B$.

Observe that $\XX(\ul\alpha,\ul\beta^{-1}) = \XX(\ul\alpha)\times_{G/B} \XX(\ul\beta^{-1})$, and we can identify $\pr_0^{-1}(e) \cap \pr_{k+\ell}^{-1}(w)$ with $X(\ul\alpha) \times_{G/B} wX(\ul\beta)$.  (One must pay attention to how $\XX(\ul\beta^{-1})$ and $wX(\ul\beta)$ map to $G/B$ in the fiber products.)  Via the restriction of $\pr_k$, the open subset $X(\ul\alpha)^\circ \times_{G/B} wX(\ul\beta)^\circ$ maps isomorphically to the dense open subset $X(u)^\circ \cap w\cdot X(v)^\circ$.  On the other hand, we can also identify $\pr_0^{-1}(e) \cap \pr_{k+\ell}^{-1}(w)$ with $\phi^{-1}(w) \subseteq X(\ul\alpha,\ul\beta^{-1})$, which is smooth and irreducible by Lemma~\ref{l.smooth}.  \qed
\end{proof}

Combining this with the criterion for dense $B$-orbits (Proposition~\ref{p.dense}), we obtain the following.

\begin{corollary}
The Bott-Samelson variety $X=X(\ul\alpha,{\ul\beta^{-1}})$ has a dense $B$-orbit if $X(u) \cap w\cdot X(v)$ is a torus orbit closure in $G/B$.
\end{corollary}

\appendix
\section{Bott-Samelson varieties are log Fano}\label{s.logfano}

Let $X=X(\ul\alpha)$ be the Bott-Samelson variety corresponding to an arbitrary sequence of simple roots $\ul\alpha=(\alpha_1,\ldots,\alpha_d)$.  Define roots $\gamma_1,\ldots,\gamma_d$ by
\[
  \gamma_i = s_d s_{d-1} \cdots s_{i+1}(\alpha_i),
\]
and write $r_i = \langle \rho, \gamma_i^\vee \rangle$, where $\rho$ is the sum of the fundamental weights.  If the word $\ul\alpha$ is reduced, then all the roots $\gamma_i$ are positive and $r_i\geq 1$ for each $i$, but in general the $r_i$'s may be negative.  The anticanonical divisor has a well-known expression, as
\[
 -K_X = \sum_i (r_i + 1) X_i
\]
(see, e.g., \cite[Proof of Proposition~10]{mr}).  

A pair $(Y,D)$ of a normal irreducible variety and an effective $\Q$-divisor is \define{Kawamata log terminal (klt)} if $K_Y+D$ is $\Q$-Cartier, and for all proper birational maps $\pi\colon Y' \to Y$, the pullback $\pi^*(K_Y+D) = K_{Y'}+D'$ has $\pi_*K_{Y'}$ and $\lfloor D' \rfloor \leq 0$.  The pair is \define{log Fano} if it is klt and $-(K_Y+D)$ is ample.

If $Y$ is already nonsingular and $D$ has normal crossings, the pair $(Y,D)$ is log Fano if and only if $\lfloor D \rfloor =0$ and $-(K_Y+D)$ is ample.

We will see that the pair $(X,\Delta)$ is log Fano, for a suitably chosen divisor $\Delta = \sum_i \epsilon_i X_i$, independent of the characteristic of the ground field.  First we need an easy lemma.

\begin{lemma}
There are positive integers $a_1,\ldots,a_d$ such that the divisor $A = \sum_i a_i X_i$ is ample.
\end{lemma}

\begin{proof}
This is almost obvious, but for concreteness here is a recipe for choosing the $a_i$.  Recall that $A$ is ample iff $\OO(A) = \OO(n_1,\ldots,n_d)$ with all $n_i>0$ \cite{lt}.  We have $n_d = a_d$, so this coefficient is positive.  Suppose $a_d,\ldots,a_{k+1}>0$ but $a_k\leq 0$, and let $A' = \sum_i a_i' X_i$ be such that $\OO(A') = \OO(n_1,\ldots,n_k-a_k+1,n_{k+1},\ldots,n_d)$.  Since the change of basis between the $\OO(1)$-basis and the divisor basis is uni-triangular, we have $a_k'=1$ and $a_i' = a_i$ for $k+1\leq i\leq d$.  Replace $A$ with $A'$ and continue inductively.\qed
\end{proof}

Given an ample divisor $A = \sum_i a_i X_i$ as in the lemma, choose a positive integer $M$ greater than all $a_i$, and let $\epsilon_i = 1 - a_i/M$.

\begin{theorem}\label{t.logfano}
The pair $(X,\Delta)$ is log Fano, where $\Delta = \sum_i \epsilon_i X_i$.
\end{theorem}

\begin{proof}
From the definition, it is clear that $\lfloor \Delta \rfloor = 0$, so we must check that $-K_X-\Delta$ is ample.  This is equal to
\[
  \sum_i (r_i + \epsilon_i) X_i = R + \frac{1}{M}A,
\]
where $R$ is the divisor associated to $\OO_d(-\rho)$, and $A$ is the ample divisor of the lemma.  Since $R$ is pulled back from an ample divisor on $G/B$, it is nef; therefore $R+(1/M)A$ is ample.\qed
\end{proof}

\begin{remark}
The theorem is an instance of the general fact that when $Y$ is a smooth variety such that $-K_Y = N+E$, with $N$ nef and $E = E_1 + \cdots + E_r$ a simple normal crossings divisor supporting an ample divisor, one can find a $\Q$-divisor $D$ supported on $E$ such that $(Y,D)$ is log Fano.  As with the analogous fact about Schubert varieties \cite{as}, our point here is to make this choice of $D$ explicit.
\end{remark}



\end{document}